\documentclass{amsart}%{amsart}
\usepackage[leqno]{amsmath} %centertags
\usepackage{amsfonts}%
\usepackage{amssymb}%
\usepackage{amsthm}%
\usepackage{verbatim}

\theoremstyle{plain}
\newtheorem{theo}{Theorem}[section]
\newtheorem{thm}{Theorem}[section]

\newtheorem{lema}[theo]{Lemma}

\theoremstyle{definition}
\newtheorem{defi}[theo]{Definition}

\theoremstyle{remark}

\begin{document}

\title{Constructive decomposition of a function of two variables as a sum of functions of one variable}

\author{Eva Trenklerov\'a}

\address{Department of Computer Science, Faculty of Science, 
P.J. \v{S}af\'arik University, Jesenn\'a 5, 04001 Ko\v{s}ice, Slovakia}
%\curraddr{Department of Mathematics and Statistics,
%Case Western Reserve University, Cleveland, Ohio 43403}
\email{eva.trenklerova@upjs.sk}
\thanks{The author was supported by grants VEGA 1/3002/06 and VEGA 1/3128/06}

% General info

% 26B40 RealFunctions;Functions of several variables;Representation and superposition of functions 
% 54C30 General Topology,  Maps and general types of spaces defined by maps,
%       Real-valued functions
% 54F99 General Topology, Special properties,  Spaces of dimension $\leq 1$;other
% 54F50 General Topology, Special properties,  Spaces of dimension $\leq 1$; curves, dendrites

% 05C90 Combinatorics, Graph theory, Applications
% 54C25 General Topology,  Maps and general types of spaces defined by maps, embedding
% 68U05 Computer science, Computing methodologies and applications, Computer graphics; computational geometry [See also 65D18]
\date{June 4, 2007}


\begin{abstract}	
Given a compact set $K$ in the plane, which does not contain any triple of points forming a vertical
and a horizontal segment, and a map $f\in C(K)$,
we give a construction of functions $g,h\in C(\mathbb R)$ such that
$f(x,y)=g(x)+h(y)$ for all $(x,y)\in K$. This provides a constructive proof of a part
of Sternfeld's theorem on basic embeddings in the plane. In our proof the set $K$
is approximated by a finite set of points.
\end{abstract}
\maketitle

%\renewcommand{\theenumi}{\arabic{enumi}}

%\include{2007_01_19_Obecne}
%  send to PROCEEDINGS!!!!!!!
% http://www.ams.org/journals/proc/aboutproc.html

\newcommand{\cesta}[4]{{#1}_{#2} {#1}_{#3}\ldots {#1}_{#4}}
\def\:{\mkern1mu\colon}
\newcommand{\fun}[3]{\mbox{${#1} \: {#2} \to {#3}$}}
\newcommand{\ww}[1]{w_{#1}}
\newcommand{\wplus}{\mathbf w_+}
\newcommand{\wminus}{\mathbf w_-}

\newcommand{\V}[1]{V^{#1}}
\newcommand{\Vplus}[1]{V_+} %{pozor, kolizne nazvy ASI}
\newcommand{\Vminus}[1]{V_-}

\newcommand{\Kx}{K_x}%{K_x}
\newcommand{\Ky}{K_y}%{K_y}
\newcommand{\p}{p}
\newcommand{\q}{q}

\newcommand{\gx}{g}%{g_x}
\newcommand{\hy}{h}%{h_y}
\newcommand{\gxe}{g'}%{g^\varepsilon_x}
\newcommand{\hye}{h'}%{h^\varepsilon_y}

\newcommand{\kmn}{k^{m_n}}

\newcommand{\Gx}{G^n}
\newcommand{\Hy}{H^n}

\newcommand{\pom}{\gamma}

\newcommand{\vare}{\varepsilon}

\newcommand{\Rset}{\mathbb R}
\newcommand{\Nset}{\mathbb N}
\newcommand{\card}{\mathrm{Card}}
\newcommand{\Rdva}{\mathbb R^2}
\newcommand{\RR}{\mathbb R}
%-----------

\newcommand{\aarray}[4]{\{{#1}_{#2}, {#1}_{#3},\ldots ,{#1}_{#4}\}}

%skratky
\newcommand{\dfe}{\delta_f(\varepsilon)}
\newcommand{\ve}{\varepsilon}

\section{Introduction}

An embedding $\fun \varphi K {\mathbb R^k}$
of a compactum (compact metric space) $K$ in the $k$-dimensional Euclidean space $\mathbb R^k$
is called a {\em basic embedding\/} provided that for each continuous real-valued function $f\in C(K)$, there exist
 continuous real-valued functions of one real variable  $g_1,\ldots,g_k \in C(\mathbb R)$ such that
$f(x_1,\ldots,x_{k}) = g_1(x_1)+\ldots+g_{k}(x_{k})$ for all points
$(x_1,\ldots,x_{k})\in \varphi(K)$. We also say, that the set $\varphi(K)$ is {\em basically embedded\/} in 
$\mathbb R^k$.

The question of the existence of basic embeddings was already implicitly
contained in Hilbert's 13th problem~\cite{hilbert}: Hilbert conjectured that not all continuous functions
of three variables were expressible as sums and superpositions of continuous functions
of a smaller number of variables.

Ostrand~\cite{ostrand} proved that each $n$-dimensional compactum 
can be basically embedded in $\mathbb R^{2n+1}$ for $n\geq 1$.
His result is an easy generalization of results of Arnold~\cite{ArnoldHilbert, ArnoldHilbertOne} and Kolmogorov~\cite{kolmogorovOne,kolmogorov}.

Sternfeld~\cite{sternfeld1} proved that the parameter $2n+1$ is the best possible in a very strong sense:
namely, that no $n$-dimensional compactum can be basically embedded in $\mathbb R^{2n}$ for $n\geq 2$.
Ostrand's and Sternfeld's results thus characterize compacta basically embeddable in $\mathbb R^k$ for $k\geq 3$.
Basic embeddability in the real line is trivially equivalent to embeddability.
The remaining problem of the characterization of compacta basically embedded in
  $\mathbb R^2$ was already raised by Arnold~\cite{ArnoldProblemSix}
%This problem, raised already by Arnold..., 
and solved by Sternfeld~\cite{sternfeld2}:

\begin{thm}[Sternfeld]\label{v:SternfeldPlane}
Let $K$ be a compactum and let $\fun \varphi K {\mathbb R^k}$ be an embedding. Then

(B) $\varphi$ is a basic embedding 

 if and only if

(A) there exists an $m\in \mathbb N$ such that the set $\varphi(K)$ does not contain an 
array of length $m$.
\end{thm}

\begin{defi} An {\em array \/} is a sequence of points $\{z_i\}_{i\in I}$ in the plane,
where $I=\{1,2,\ldots,m\}$ or $I=\mathbb N$, such that for each $i$:
\begin{itemize}
\item $z_i\ne z_{i+1}$ and $[z_i;z_{i+1}]$ is a segment parallel to one of the coordinate axes and
\item the segments $[z_i;z_{i+1}]$ and $[z_{i+1};z_{i+2}]$ are mutually orthogonal.
\end{itemize}
If $I=\{1,2,\ldots,m\}$ then the length of the array is $m-1$.
\end{defi}

Using the geometric description (A),  Skopenkov~\cite{skopenkov} gave a characterization of continua basically embeddable in the 
plane by means of forbidden subsets resembling
Kuratowski's characterization of planar graphs. 
%When his result is restricted 
%to finite graphs it states that
%a finite graph is basically embeddable in the plane
%if and only if it does not contain a subset homeomorphic to any of the following:
%a circle a pentod and a cross with branched ends. 
In a similar way Kurlin~\cite{Kurlin00} characterized finite graphs
basically embeddable in $\mathbb R\times T_n$, where $T_n$ is
a star with $n$-rays.
Repov\v s and \v Zeljko~\cite{Zeljko} proved a result concerning the smoothness of functions
in a basic embedding in the plane.

Sternfeld's proof of  the equivalence (A) $\Leftrightarrow$  (B) 
is not direct but
uses a reduction to linear operators. In particular it is not constructive. 
It is therefore desirable to find a straightforward, constructive proof
which will consequently provide an elementary proof
of Skopenkov's and Kurlin's characterizations.
%
%?Kde sa spomina g,h?!!!
A constructive proof of  (B) $\Rightarrow$  (A)  is given
in~\cite{EvaNeza}. 
%Given a compact subset $K$ of the plane  such that
%the set $E^n(K)$ is nonempty  for every $n$ in $\Nset$,
%the authors give a construction of a function $f$ in $C(K)$ which is
%not expressible in the form $f=g+h$.% with $g,h$ in $C(\mathbb R)$.

In this paper we give such an elementary construction, thus proving the implication (A) $\Rightarrow$ (B) 	
provided that $m=2$:

\begin{thm}\label{v:Basic} Let $\fun \varphi K {\mathbb R^2}$ be an embedding of a compactum
$K$ in the plane such that the set $\varphi(K)$ does not contain an array of length two.  
Then for every function $f \in C(\varphi(K))$
 there exist functions  $g,h\in C(\mathbb R)$ such that $f(x,y)=g(x)+h(y)$
for all points $(x,y)\in \varphi(K)$.
%, that is the embedding of  $K$ in $\Rdva$ is basic.
\end{thm}

The main part of our proof consists in
finding an approximate decomposition of a given function $f$ as $g+h$.
The functions $g$, $h$ are defined on a finite approximation $\V n$ of $\varphi(K)$.
Then they are linearly
extended to $\mathbb R$.
Apart from two steps, where we asset the existence of certain constants,
this part of the proof is constructive.
The existence of an exact decomposition follows by an elementary
iterative procedure.

Until now, no constructive decomposition of $f$ as $g+h$ on compacta in the plane satisfying
(A) of Theorem~\ref{v:SternfeldPlane} has been found,
not even in the simplest case, when the compactum satisfies (A) with $m=2$.

Our result resembles 
representation theorems of Arnold~\cite{ArnoldHilbert, ArnoldHilbertOne}, Kolmogorov~\cite{kolmogorovOne,kolmogorov} and Ostrand~\cite{ostrand}. The proofs are similar in that we also construct a sequence of finite families of squares. But different from these proofs, where the squares (or cubes in higher dimensions)
are connected only with the dimension of the set in question,
here the squares mimic the property that the set does not contain an array of length two.

% see version 2007_05_23 for more

In the paper~\cite{Zeljko} the authors give the decomposition for finite graphs basically embedded in the plane:
according to the results of~\cite{skopenkov, RepovsOpen98}, a finite graph can be basically embedded in the plane
if and only if it can be embedded in a special graph $R_n$ for some $n$. 
The authors of~\cite{Zeljko} inductively define an embedding $\fun \varphi {R_n} {\mathbb R^2}$.
For a given function $f\in C(\varphi(R_n))$ they define the maps $g,h\in C(\mathbb R)$ inductively again,
starting from a well chosen subset of $\varphi(R_n)$. 
%(In fact, the main result of the article is, that 
%if $f$ is in $C^1(\varphi(R_n))$ then the functions $g,h$ constructed in this way are in $C^1(\mathbb R)$.)

Although the sets we are dealing with do not contain arrays of length two, they can be still
\lq\lq arbitrarily bad\rq\rq. % and we cannot say anything about their structure. 
In particular we are not
able to choose a suitable subset to start the construction on. %More notes are in Section~\ref{s:Ideas}.
%
\begin{comment}
We can formulate this in the sense of the note on what a good measure for a \lq\lq bad\rq\rq\ is:
It turns out, that the maximal length of an array a compactum in the plane contains is not
 a good characterization of \lq\lq badness\rq\rq -- compact sets can be so  \lq\lq bad\rq\rq, that the
 maximal length of arrays they contain is not particularly important.
(More about what \lq\lq bad\rq\rq\ means  is given at the beginning of Section~\ref{s:Ideas}.)

It turns out, that the greatest problem is connectivity.
\end{comment}
%Let $K$ be a compact subset of the plane  such that $E^n(K)=\emptyset$ for some $n$.
%If $K$ is a finite graph  with $E^n(K)=\emptyset$ for some $n$, then it is possible to
%construct functions $g$ and $h$ with $f=g+h$, for every function $f$ in $C(K)$
% (see~\cite{Zeljko}). 
%However, if $K$ is an arbitrary compactum, then
%no construction of functions $g$ and $h$ was known, even if $E(K)=\emptyset$.
%In this paper we present 
%such a construction for sets $K$ with empty $E(K)$.
%We hope that the proof can be generalized for sets $K$
%with $E^n(K)=\emptyset$ for arbitrary $n$. 
%This proof is in the line of the
%representation theorems of Arnold, Kolmogorov and Ostrand. Likewise
%we give a decomposition of a continuous function of more variables
%by means of continuous functions of one variable.
%This proof may be considered as a special case
%of a constructive proof of the general case.
%
Thus it turns out,  that even if a set $\varphi(K)\subseteq \mathbb R^2$ satisfies the simplest
version of condition (A), 
a constructive decomposition of a function $f\in C(\varphi(K))$ is a non-trivial problem.
%this turns out to be a non trivial problem...
We believe, that the proof can be modified  to obtain a constructive proof of
the implication (A) $\Rightarrow$ (B) for an arbitrary $m \in \mathbb N$. %NEZA: stronger than believe...????

The author would like to thank Du\v san Repov\v s and Arkadyi Skopenkov for the inspiration for
this paper, Lev Bukovsk\'y for support and especially Ne\v za Mramor-Kosta for endless conversations on the topic and
invaluable advice.

The author would also like to thank the anonymous referee for  valuable suggestions and comments.

\section{Notation and conventions}\label{s:Notation}

Throughout the text we fix an embedding $\fun \varphi K {\mathbb R^2}$ of a compactum
$K$ in the plane such that the set $\varphi(K)$ does not contain an array of length two. 
For simplicity of notation we identify the set $K$ and its homeomorphic image $\varphi(K)$
and we speak about a set $K\subseteq \mathbb R^2$.
%Let us fix a compactum $K\subseteq \mathbb R^2$ which does not contain an array
%of three points, 
Let $f$ be from $C(K)$ and $\vare>0$ be the desired approximation constant.
Since $f$ is continuous on the compact set $K$, it is uniformly continuous there. 
Therefore, there exists a positive real $\delta=\delta(K,f,\vare)>0$ such that for all points $z,z'\in K$ if
$|z-z'|<\delta$ then $|f(z)-f(z')|<\varepsilon.$ We fix this $\delta$ as well.

The distance in $\mathbb R^2$ is defined as
$|(x,y)-(x',y')|=\max\{|x-x'|, |y-y'|\}$ for $(x,y),(x',y')\in \mathbb R^2$. By $\fun {\p,\q} {\mathbb R^2} {\mathbb R}$ we denote  the vertical and horizontal orthogonal projections:
$\p(x,y)=x$, $\q(x,y)=y$. 

\section{Idea of the proof and the main statements}\label{s:Ideas}
Our proof of Theorem~\ref{v:Basic} mimics the following construction of the functions $\gx,\hy$ which
works for  certain types of sets $K$ (for example graphs, considered in~\cite{Zeljko}). 
 Denote by $\Kx$ the set of all points $(x,y)\in K$ which have a neighbor in the
vertical direction in $K$, i.e. $\Kx=\{(x,y)\in K|\exists (x,y')\in K, y\ne y'\}.$ Similarly define $\Ky.$
Assume that both sets $\Kx$ and $\Ky$ are closed.

Since $K$ does not contain an array of length two, the functions $\p$ and $\q$ are injective
on $\Ky$ and $\Kx$, respectively, and the sets $\Kx$ and $\Ky$ are disjoint.
 %in particular $\p$ is injective on $K-K_x$ and  $\q$ is injective on $K-\Ky$. 
For each point $x\in \p(\Ky)$ let $\gx(x)=f(x,\p^{-1}(x))$ and for each point $x\in \p(\Kx)$ let $\gx(x)=0$. 
Extend $\gx$ continuously to $\mathbb R$.
The function $\hy$ is defined in the following way: for each point $y\in q(K)$  pick an arbitrary point
$(x,y)\in K$ and let $\hy(y)=f(x,y)-\gx(x)$. It is easily seen that $\hy$ is continuous. 
We extend $\hy$  to $\mathbb R$.
%
%This construction will be mimicked in the proof of Theorem~\ref{v:GHOnGraph}.

We have defined the functions $\gx$ and $\hy$ first on the sets $\Kx$ and $\Ky$. In general
these sets are not closed, and the set $K$ can be so \lq\lq bad\rq\rq\ that 
we cannot find a suitable set to begin the definitions.

The main part of the proof consists of three steps. 

\subsection*{Step 1} For each $n$ we construct the set $\V n$ approximating $K$.
Consider the lattice 
$(i/2^n,j/2^n)$ with $i,j\in \mathbb Z.$
For each square $\left[ {i}/{2^n};{(i+1)}/{2^n} \right) \times \left[ {j}/{2^n};{(j+1)}/{2^n} \right)$ 
which intersects the set $K$ we choose one point
from the intersection of this square with the set $K$.
The set $\V n$ consists of all the chosen points.

We shall call a segment $[(u_1,v_1);(u_2,v_2)]$ given by
 a pair of points $(u_1,v_1)$, $(u_2,v_2)\in \V n$  {\em almost vertical in $\V n$\/}
if $|u_1-u_2|<{2}/{2^n}$
and {\em almost horizontal in $\V n$\/}
if  $|v_1-v_2|<{2}/{2^n}$.

Points which are the ends of almost vertical or horizontal segments 
are \lq\lq near\rq\rq\ to each other in the vertical or horizontal direction,
respectively.

Arrays in the set $K$ are \lq\lq deformed\rq\rq\ in the approximating finite set $\V n$
so we generalize the notion of an array, in the following way.

\begin{defi}\label{d:AlmostArray} A sequence of pairwise different points $\{w_1,w_2,\ldots,w_m\}$ from the set $\V n$
is said to form an {\em almost array in $\V n$\/} if each pair of consecutive points 
$w_i$, $w_{i+1}$ forms an almost vertical or horizontal segment
 in $\V n$. The {\em length\/} of the almost array is defined to be $m-1$.
\end{defi}
%????The condition concerning $f$ is formalized by stating that $1/2^n<\delta$.???

\subsection*{Step 2}

In this step we define approximations $\Gx$ and $\Hy$ of $\gx$ and $\hy$ on the set $\V n$.
This step contains the major part of the proof and consists of proving three statements.

If the distance of two points from $\V n$ is smaller than $\delta$
then the difference of $f$ between them is bounded by $\vare$. As we are approximating up to $\vare$,
such points are \lq\lq almost the same\rq\rq\ for us. 

\begin{defi}\label{d:ShortLong} A segment $[z_1;z_2]$ with $z_1,z_2\in \Rdva$ is said to be
{\em long\/} if $|z_1-z_2|\geq \delta$ and it is said to be 
{\em short\/} if $|z_1-z_2|< \delta$.
\end{defi}

The set of all points from $\V n$ which are the ends of
the long almost vertical segments in $\V n$ is the analogue of the set $\Kx$ and the set of all points from $\V n$ which are the ends of
the long almost horizontal segments in $\V n$ is the analogue of the set $\Ky$.

\begin{theo}\label{v:GOnGraph}  There exists an $n_1$ such that for all $n\geq n_1$
a function
$\fun \Gx {\p(\V n)} {\mathbb R}$ satisfying the following requirements exists.
\begin{enumerate}
  \item\label{v:GOnGraph_Jedna} 
  \begin{enumerate}
   \item\label{v:GOnGraph_JednaA}  $|\Gx(u_1)-\Gx(u_2)|\leq 3\varepsilon$ for each short segment $[(u_1,v_1);(u_2,v_2)]$ in $\V n$
       \item\label{v:GOnGraph_JednaB}  $|\Gx(u)-f(u,v)|\leq 2\vare$ for each $(u,v)\in \V n$ which is the end of a long almost horizontal segment in $\V n$ 
    \item\label{v:GOnGraph_JednaC} $|\Gx(u)|\leq \vare$ for each $(u,v)\in \V n$ which is the end of a long almost vertical segment in $\V n$ 
    \end{enumerate}
   \item\label{v:GOnGraph_Dva}  $||\Gx||\leq ||f||$.
\end{enumerate}
\end{theo}

The function $\Hy$ is constructed using $\Gx$ from Theorem~\ref{v:GOnGraph}:

\begin{theo}\label{vv:GHOnGraph} %Let $f\in C(K)$, $\ve>0$, $\delta>0$ be as above.
%Let $n_0=n_0(K,l,\alpha)$ be from Theorem~\ref{v:Thm1} with
%$l=[||f||/\ve]+1$ and $\alpha=\delta/2$.
%Let $n=\max\{n_0,[-\log_2\delta]+1\}.$ 
%$n$ be such that $n\geq n_0$ and $2^n>2/\delta.$ %   
%\noindent
There exists an $n_2$ such that for all $n\geq n_2$ functions
$\fun \Gx {\p(\V n)} {\mathbb R}$ and $\fun \Hy {\q(\V n)} {\mathbb R}$ satisfying the following requirements exist.
\begin{enumerate}
\item\label{vv:GHOnGraph_Jedna} 
   \begin{enumerate}
     \item\label{vv:GHOnGraph_JednaA} $|f(u,v)-\Gx(u)-\Hy(v)|\leq 4\varepsilon$ for each $(u,v)\in \V n$
  	 \item\label{vv:GHOnGraph_JednaB}  
  	  $|\Gx(u_1)-\Gx(u_2)|\leq 3\varepsilon$ 
  	  for each segment $[(u_1,v_1);(u_2,v_2)]$ which is almost vertical in $\V n$ 
  	 \item\label{vv:GHOnGraph_JednaC}  $|\Hy(v_1)-\Hy(v_2)|\leq 12\varepsilon$ 
  	  for each segment $[(u_1,v_1);(u_2,v_2)]$ which is almost horizontal in $\V n$ 
    \end{enumerate}
  \item\label{vv:GHOnGraph_Dva} $||\Gx||\leq ||f||$, $||\Hy||\leq 2||f||.$ 
\end{enumerate}
\end{theo}

In order to construct the function $\Gx$ from Theorem~\ref{v:GOnGraph} we need the following lemma.
Its proof is based on the fact that $K$ is compact and does not contain an array of length 2. It will
be used with $l=[||f||/\vare]$.

\begin{lema}\label{l:AlmostArray} For each $l\in \mathbb N$ there exists an $n_0$ such that for all $n\geq n_0$ the following holds:
if $\{w_1,\ldots,w_k\}$ is an almost array in $\V n$ and $w_1$ is the end of a long almost vertical segment in $\V n$
and $w_k$ is the end of a long almost horizontal segment in $\V n$, then the length of the almost array
is at least $l$, i.e. $k-1\geq l$.
\end{lema}

%\begin{theo}\label{v:AlmostArrayA} There exists $n_0$ such that for all $n\geq n_0$ the following holds:
%If $w_0,\ldots,w_k$ is an almost array in $\V n$ and $w_0$ is the end of a long almost vertical segment in $\V n$
%and $w_k$ is the end of a long almost horizontal segment in $\V n$ then the length of the quasi array
%is at greater than $[||f||/\vare]$, i.e. $k>[||f||/\vare]$.
%\end{theo}

\subsection*{Step 3}

The functions $\gx$, $\hy$ are obtained by linear
extensions of the functions $\Gx$, $\Hy.$

\begin{theo}\label{v:GHOnR} There exists $n_3$ such that for all $n\geq n_3$ functions
$\gx,\hy\in C(\mathbb R)$ satisfying the following requirements exist.
such that
\begin{enumerate}
	\item $|f(x,y)-\gx(x)-\hy(y)|\leq 20\varepsilon$ for all points $(x,y)\in K$
	\item $||\gx||\leq ||f||$, $||\hy||\leq 2||f||$.
\end{enumerate}
\end{theo}

\section{Proof of the main statement}

\begin{proof}[Proof of Theorem~\ref{v:Basic}]
Assuming that
the above three steps have been accomplished, the  statement follows immediately from
Theorem~\ref{v:GHOnR} and Theorem~\ref{v:GHExact} below. 
\end{proof}

To make the proof clearer though, we explicitly describe our construction.

According to Lemma~\ref{l:AlmostArray}
there exists an $n_0$ such that for all $n\geq n_0$ every almost array in $\V n$
starting in a long almost vertical segment and ending in a long almost horizontal segment has length
at least $[||f||/\vare]$. We define a constant $N=\max\{n_0, -[\log_2 \delta]\}$ and
let the lower bounds  $n_1, n_2$ and $n_3$ from Theorems~\ref{v:GOnGraph}, \ref{vv:GHOnGraph} and \ref{v:GHOnR} be all equal to $N$. We take an arbitrary $n\geq N$.

We construct $\V n$ and the corresponding function $\fun \Gx {\p(\V n)} {\mathbb R}$ from Theorem~\ref{v:GOnGraph}
which approximates $\gx$.
Using $\Gx$, we define the function $\fun \Hy {\q(\V n)} {\mathbb R}$ from Theorem~\ref{vv:GHOnGraph} 
thus obtaining the approximate of $\hy$.  The functions $\Gx$ and $\Hy$
are extended as piecewise linear functions on $\mathbb R$ thus obtaining 
the functions $\gx$ and $\hy$ from
Theorem~\ref{v:GHOnR}. 
Applying Theorem~\ref{v:GHExact} below we obtain the exact decomposition.

\begin{theo}[implication (b) 
$\Rightarrow$ (c) of Theorem 4.13 in \cite{rudin}]\label{v:GHExact} Let $X\subseteq \mathbb R^2$ be an arbitrary compact subset of the plane.
Assume that there exists a positive integer $k\in \mathbb N$ such that for each function $f\in C(X)$
and each positive real $\varepsilon>0$ there exist functions $\gxe,\hye\in C(\mathbb R)$
such that
\begin{enumerate}
\item $|f(x,y)-\gxe(x)-\hye(y)|\leq \varepsilon$
for all points $(x,y)\in X$ 
\item $||\gxe||\leq k||f||$, $||\hye||\leq k||f||$.
\end{enumerate}
Then there exist functions $\gx,\hy\in C(\mathbb R)$ such that
$f(x,y)=\gx(x)+ \hy(y)$
for all points $(x,y)\in X$.
\end{theo}

\section{Proofs of the statements}

Let us give the proofs of the statements in the order in which we use them to prove Theorem~\ref{v:Basic}.

\begin{proof}[Proof of Lemma~\ref{l:AlmostArray}] 
First, let us note the following. Let
$\{[w_1^{n}; w_2^{n} ]\}_{n=1}^\infty$ be a sequence 
where each
 $[w_1^{n}; w_2^{n}]$ is an almost vertical or an almost horizontal segment in  $\V n.$
Then, since $K$ is compact, there is a subsequence $\{[w_1^{m_n}; w_2^{m_n} ]\}_{m_n}$ %$\{m_n\}_{n\in \mathbb N}$
such that both $w_1^{m_n} \to w_1 \in K$
and $w_2^{m_n} \to w_2 \in K$ as $n\to \infty.$
Evidently, either $w_1=w_2$ or $[w_1;w_2]$ is a segment
parallel to one of the coordinate axes.
Moreover, if $|\p(w^{m_n}_1)-\p(w^{m_n}_2)|\geq \delta$
for each $n$ then  $|\p(w_1)-\p(w_2)|\geq \delta$ and $[w_1;w_2]$ is a segment 
parallel to the $x$ axis.
Similarly for the projection $\q$.

%The proof of the theorem is indirect. 
Assuming that the statement is not true, we will show
that $K$ contains an array of length two.
So, assume that for some $l_0$  there exists an increasing sequence $\{m_n\}_{n=1}^\infty$
of integers 
such that each set
$\V {m_n}$ contains an almost array $\{w^{m_n}_1, w^{m_n}_2,\ldots ,w^{m_n}_{\kmn}\}$
as in the statement, but its length $\kmn-1$ is smaller than $l_0$, so $\kmn\leq l_0$. 
This implies that infinitely many of the numbers $\kmn$ are the same.
Without loss of generality we may assume that $\kmn= l_0$ for all $n$.
For each $i$, the point $w^i_1$ is the end of a long almost vertical segment in $\V n$ denoted by $[w^i_{0};w^i_1]$
and the point $w^i_{l_0}$ is the end of a long almost vertical segment in $\V n$, denoted by $[w^i_{l_0};w^i_{l_0+1}]$.

It follows, that there exist limit points 
$w_{0},w_1,\ldots,w_{l_0+1} \in K$ such that either $w_i=w_{i+1}$ or
$[w_i;w_{i+1}]$ is a segment parallel to one of the coordinate axes for each $i$. In particular,
$[w_{0};w_1]$ is a vertical segment and $[w_{l_0};w_{l_0+1}]$ is a horizontal segment.
Therefore, the set $\{w_{0},w_1,\ldots,w_{l_0+1}\}\subseteq K$ contains an array of length two.
\end{proof}

%--------------------------------------------------------------
%                   begin proof
%--------------------------------------------------------------
\begin{proof}[Proof of Theorem~\ref{v:GOnGraph}]
Denote $F=[{||f||}/{\varepsilon}]$.
%\begin{equation}\label{e:ChoiceOfL}
%F=\left[\frac{||f||}{\varepsilon}\right]. \end{equation} 
%
Let $n_1$ be equal to $n_0=n(l)$ from Lemma~\ref{l:AlmostArray} which corresponds
to $l=F$. Take an arbitrary $n\geq n_1$.

\renewcommand{\theenumi}{\roman{enumi}}
First we define a function 
 $\fun \pom {\V n} \RR$ such that
\begin{enumerate}
  \item\label{v:pomS_Jedna}
  \begin{enumerate}
   \item\label{v:pomS_JednaA} $|\pom(w_1)-\pom(w_2)|\leq \varepsilon$ for each
    short segment $[w_1;w_2]$ 
     \item\label{v:pomS_JednaB}  $|f(w)-\pom(w)|\leq \varepsilon$ for each $w$ which is the end of a long almost   horizontal segment
     \item\label{v:pomS_JednaC} $\pom(w)= 0$ for each $w$ which is the end of a long almost vertical segment  
  \end{enumerate}
  \item\label{v:pomS_Dva} $||\pom||\leq ||f||.$ 
\end{enumerate}
\renewcommand{\theenumi}{\alph{enumi}}

Second, we define $\Gx$ using $\pom$; we shall have roughly $\Gx(u)= \pom(u,v) \pm \vare$ for all $(u,v)\in \V n$. 
To construct $\pom$, we define two abstract graphs with vertices  from $\V n$.
They are not embedded in the plane.

Assume, that for each $i=-F, -F+1,\ldots, F$ there exists a point $w\in \V n$
such that $[f(w)/\vare]=i$. If this is not the case then we add a new point $z \in \Rdva$ to $\V n$
for each $i$ for which no such point
exists. Formally we consider
it as the end of a long almost horizontal segment, and let $f(z)=i\vare$.
The point is added so that its distance from each point from $\V n$ is greater than $\delta$.
%The construction would be the same.
%The additional points are used only in this construction. %The construction is the same.

Let $\Vplus n$ be the set of all points $w\in \V n$ with $f(w)\geq 0$ with one  vertex $\wplus$ added:
$$\Vplus n = \{w\in \V n\mid  f(w)\geq 0\} \cup \{\wplus\}.$$
We define $f(\wplus)=(F+1)\vare$.
Let $\Vminus n$ be the set of all points $w\in \V n$ with $f(w)< 0$ with one vertex $\wminus$ added:
$$\Vminus n = \{w\in \V n\mid  f(w)< 0\} \cup \{\wminus\}$$
and let $f(\wminus)=(-F-1)\vare$.

The edge set $E(\Vplus n)$ consists of edges
\begin{itemize}
\item $w_1w_2$ where $[w_1;w_2]$ is a short segment
\item $w_1w_2$ where both $w_1$ and $w_2$ are the ends of long almost horizontal segments in $\V n$ and
$[f(w_1)/\vare]-[f(w_2)/\vare]=1$
%$[f(w_1)/\vare]=i$, $[f(w_2)/\vare]=(i-1)$ for some $i$
%
\item $\wplus w$ where $w$ is the end of a long almost horizontal segment in $\V n$ and
$[f(w)/\vare]=F$.
\end{itemize}

Evidently 
\begin{equation}\label{e:EachEdge} |f(w_1) - f(w_2)|\leq \varepsilon
\end{equation} for each edge $w_1w_2 \in E(\Vplus n)$.
The edges $E(\Vminus n)$ are defined analogously.

%We define $\pom$ on $\Vplus n$.
%
%%%%%%%%%%%%%%%%%%%%%%%%%%%%%%%%%%%%%%%%%%%%%%%%%%
%
Let $\fun d {\Vplus n} {\{0,1,2,\ldots\}}$ be the
function, assigning to each vertex which is connected to $\wplus$ by a path in $E(\Vplus n)$
its distance from $\wplus$, and assigning to each other vertex the value 0.
For each vertex $w\in \Vplus n$ let
\begin{equation}\label{e:ThirdPartDefG} \pom(w) = \max\left\{
(F - d(w) + 1)\vare, 0 \right\}.\end{equation}
Analogously we define the function $\pom$ on $\Vminus n$.

Let us show that the function $\pom$  satisfies~(\ref{v:pomS_Jedna}) and (\ref{v:pomS_Dva}).

(\ref{v:pomS_JednaA}) Let $[w_1;w_2]$ be a short segment in $\V n$.
If both $w_1$ and $w_2$ are in $\Vplus n$, or both $w_1$ and $w_2$ are in $\Vminus n$ then
$|\pom(w_1)-\pom(w_2)|\leq \varepsilon$ follows directly from the definition of $\pom.$
So, let $w_1\in \Vplus n$ and $w_2\in \Vminus n$. Using
(\ref{e:EachEdge}), (\ref{e:ThirdPartDefG}), by induction
on the distance from the vertex $\wplus$  and analogously,
by induction on the distance from $\wminus$  we can show that
\begin{equation}\label{e:WhereIsG}
\begin{array}{ll}
\phantom{f(),} 0\leq  \pom(w)  \leq f(w) & \text{for all }\ w\in \Vplus n \\%\pom(w) \in [0,f(w)],\  \text{ for all }\ w\in \Vplus n \\
f(w)\leq  \pom(w) \leq 0 &\text{for all }\  w\in \Vminus n
\end{array}
\end{equation}

Since the segment $[w_1;w_2]$ is short we have $|f(w_1)-f(w_2)|<\varepsilon$.
So $|\pom(w_1)-\pom(w_2)| = \pom(w_1)-\pom(w_2) \leq f(w_1)-f(w_2)< \varepsilon$.

%------------------------------------------
%------------------------------------------
%------------------------------------------
(\ref{v:pomS_JednaB}) Let $w$ be the end of a long almost horizontal segment in $\V n$. 
Let  $w\in \Vplus n$ for instance. Denote $i=[f(w)/\vare]$.
%
%Since $w$ is the end of a long almost horizontal segment in $\V n$, 
By definition there is a path 
$\wplus \ww {F}\ldots \ww {i+2} \ww {i+1} w$ whose edges are in $E(\Vplus n)$.
Each of its vertices 
$\ww j$ is the end of a long almost horizontal segment in $\V n$ and $[f(\ww j)/\vare]=j$.
%The existence of such a path is ensured by the assumption on $\V n$ stated in the beginning of the construction of %$\gamma$.
So $d(w)\leq F - i+1.$ 
Hence $\pom(w) =(F-d(w)+1) \varepsilon\geq i\varepsilon > f(w) - \vare.$ 
%Hence    This observation implies that $\pom(u) - f(u) \geq -2\varepsilon$,
On the other hand, Equation~(\ref{e:WhereIsG}) implies that $\pom(w)\leq f(w)$
and (\ref{v:pomS_JednaB}) follows.

%------------------------------------------
%------------------------------------------
%------------------------------------------
%{\bf 1(c)}
(\ref{v:pomS_JednaC})
Let $w$ be the end of a long almost vertical segment in $\V n$.
Let $w\in \Vplus n$ for instance.
%We want to show that $\pom(w)=0$. 
If $w$ is not connected to $\wplus$ by a path then, by definition, $\pom(w)=0.$
 
So, let
$\wplus \ldots w$ be a path such that $d(w)$ is
equal to its length. % This path necessarily contains an edge of the form $\w_i w_3$,
%where $w_3$ is the end of a long almost horizontal segment in $\V n$.
%
Let $w'\ldots w$ be its longest subpath containing $w$, such that each of it edges
corresponds to a short
almost vertical or
almost horizontal segment in $\V n$. Then
$w'$ is the end of a long horizontal segment in $\V n$.
Thus the vertices of the subpath form an almost array
which satisfies the requirements of Lemma~\ref{l:AlmostArray}.
% Since $n\geq n_1=n_0(F)$, its length is at least $F$, that is
%$d(w)-1 \geq F>F-1$.
Hence, if we denote its length by $p$, we have $F\leq p\leq d(w)-1$.
So $\pom (w)=(F- d(w)+1)\varepsilon < \varepsilon$. On the other hand, by the definition~(\ref{e:ThirdPartDefG}) of $\pom$,
we have $\pom(w)\geq 0.$ Since the values of $\pom$ are integer multiples of $\vare$, it follows that $\pom(w)=0$.

Point~(\ref{v:pomS_Dva}) follows directly from (\ref{e:WhereIsG}).

%--------------------------------------------------------------
%                   second part
%--------------------------------------------------------------
\medskip The function $\Gx$ is constructed in the following way.
For each point $u\in \p(\V n)$ fix an arbitrary point $(u,v)\in
\V n$, and define
$\Gx(u)=\pom(u,v).$ 
Let us show that $\Gx$ satisfies
(\ref{v:GOnGraph_Jedna}), (\ref{v:GOnGraph_Dva}) from Theorem~\ref{v:GOnGraph}.

(\ref{v:GOnGraph_JednaA}) Let $[(u_1,v_1);(u_2,v_2)]$ be a short segment in $\V n$.
Let $\Gx(u_1)=\pom(u_1,v'_1)$ and let $\Gx(u_2)=\pom(u_2,v'_2)$.
Denote $w_i=(u_i,v_i)$ and $w'_i=(u_i,v'_i)$  for $i=1,2$. Since $[w_1;w_2]$ is a short segment,
by~(\ref{v:pomS_JednaA}) we have $|\pom(w_1)-\pom(w_2)|\leq \vare.$ 
The segment $[w_1;w_1']$ is almost vertical (in fact it is vertical).
If it is short then  $|\pom(w_1)-\pom(w_1')|\leq \vare$, by~(\ref{v:pomS_JednaA}).
If it is long then $\pom(w_1)=\pom(w_1')=0$, by~(\ref{v:pomS_JednaC}). The same is true for $[w_2;w_2']$.

If both segments $[w_1;w_1']$, $[w_2;w_2']$ are short then 
$|\Gx(u_1)-\Gx(u_2)|=|\pom(w_1')-\pom(w_2')|\leq 3\vare$. 
If the first one is long and the second one is short then $|\pom(w_1)-\pom(w_2)|=|\pom(w_2)|\leq \vare$
and we have
$|\pom(w_1')-\pom(w_2')|=|\pom(w_2')|\leq |\pom(w_2)|+\vare \leq \vare$. 
If both are long then $|\pom(w_1')-\pom(w_2')|=0$. 

Hence we have proved~(\ref{v:GOnGraph_JednaA}). Points (\ref{v:GOnGraph_JednaB}) and 
(\ref{v:GOnGraph_JednaC}) are proved in a similar way. Point (\ref{v:GOnGraph_Dva}) follows directly from~(\ref{v:pomS_Dva}).
%
%
%--------------------------------------------------------------
%                   DETAILED CALCULATIONS BEGIN
%--------------------------------------------------------------
\begin{comment}
(\ref{v:GOnGraph_JednaB}) Let $w=(u,v)$ be the end of a long almost horizontal segment in $\V n$.
Then $|\pom(w)-f(w)|\leq \vare$ by~(\ref{v:pomS_JednaB}). Let $\Gx(u)=\pom(u,v')$. Denote $w'=(u,v')$.

The segment $[w;w']$ is vertical. If it is short then $|\Gx(u)-f(u,v)|=|\pom(w')-f(w)|\leq |\pom(w)-f(w)|+\vare \leq 2\vare$.
If it is long then $\pom(w)=\pom(w')=0$ by~(\ref{v:pomS_JednaC}) and so $|\pom(w)-f(w)|=|f(w)|\leq \vare$.
Then we have $|\Gx(u)-f(u,v)|\leq \vare$.

(\ref{v:GOnGraph_JednaC}) Let $w=(u,v)$ be the end of a long almost vertical segment in $\V n$.
Then $\pom(w)=0$ by~(\ref{v:pomS_JednaC}). Let $\Gx(u)=\pom(u,v')$ and denote $w'=(u,v')$.

If the segment $[w;w']$ is short then $|\Gx(u)|=|\pom(w')|\leq |\pom(w)|+\vare=\vare$, by~(\ref{v:pomS_JednaA}). If it is long then $|\Gx(u)|=|\pom(w')|=0$ by~(\ref{v:pomS_JednaC}).
\end{comment}
%--------------------------------------------------------------
%                   DETAILED CALCULATIONS END
%--------------------------------------------------------------
\end{proof}

%--------------------------------------------------------------
%                  PROOF START
%--------------------------------------------------------------

\begin{proof}[Proof of Theorem~\ref{vv:GHOnGraph}] Let $n_2$ be equal to $n_1$ from Theorem~\ref{v:GOnGraph}.
Take an arbitrary $n\geq n_2$ and the corresponding function $\fun \Gx {\p(\V n)} {\mathbb R}$ from Theorem~\ref{v:GOnGraph}. Define $\fun \Hy {\q(\V n)} {\mathbb R}$ in the following way: for each
$v\in \q(\V n)$ fix a point $(u,v)\in \V n$ and let $\Hy(v)=f(u,v)-\Gx(u).$ 

The arguments showing that (\ref{vv:GHOnGraph_Jedna}) and (\ref{vv:GHOnGraph_Dva}) are satisfied are similar
to those in the proof of Theorem~\ref{v:GOnGraph}. %Additionally, one  has to use the property, that
\end{proof}

\begin{proof}[Proof of Theorem~\ref{v:GHOnR}]
Take $n_2$ from Theorem~\ref{vv:GHOnGraph} and let $n_3=\max\{n_2, -[\log_2 \delta]\}$. Then
$1/2^{n_3}\leq \delta$. Let 
$n\geq n_3$ and take functions $\fun \Gx {\p(\V n)} {\mathbb R}$ and $\fun \Hy {\q(\V n)} {\mathbb R}$ from Theorem~\ref{vv:GHOnGraph}.

Denote the points of the $\p$ projection of the set $\V n$ by $x_1,\ldots,x_k$ with
$x_i<x_{i+1}$ for all $i$. Let $\gx(x_i) = \Gx(x_i)$ for each $i$.
On each interval $[x_i;x_{i+1}]$ such that $|x_i - x_{i+1}|<1/2^{n-1}$ extend $\gx$ linearly 
between the values $\gx(x_i)$ and $\gx(x_{i+1})$. If an interval $[x_i;x_{i+1}]$
is such that $|x_i - x_{i+1}|\geq 1/2^{n-1}$ then there exists an interval of the form
$I=[j/2^n;(j+1)/2^n)$ such that $x_i<j/2^n<(j+1)/2^n<x_{i+1}$ and $\p^{-1}(I)\cap K = \emptyset$.
On the intervals $[x_i;j/2^n]$ and $[(j+1)/2^n;x_{i+1}]$ extend $\gx$ as a constant, equal to
$\gx(x_i)$ and $\gx(x_{i+1})$, respectively. On the intervals $(-\infty;x_1]$ and $[x_k;\infty)$ extend
$\gx$ as a constant as well.

Denote the points of the $\q$ projection of the set $\V n$ by $y_1,\ldots,y_l$ with
$y_j<y_{j+1}$ for all $j$. Let $\hy(y_j) = \Hy(y_j)$ for each $j$. Extend $\hy$ to $\mathbb R$
in a similar way as $\gx$.

Every point $(x,y)$ from $K$ lies in a square $S=\left[{i}/{2^n};{(i+1)}/{2^n} \right) \times \left[ {j}/{2^n};{(j+1)}/{2^n} \right)$ and there is a point $(x_i,y_j)\in S\cap \V n.$
Let $x_i \leq x$ and $y_j \leq y$ for instance. If $|x_i - x_{i+1}|<1/2^{n-1}$ then
$x_i$ and $x_{i+1}$ are the vertical projections of the ends of an almost vertical segment in $\V n$.
So, by point~(\ref{vv:GHOnGraph_JednaB}) of Theorem~\ref{vv:GHOnGraph} we have $|\Gx(x_i) - \Gx(x_{i+1})|=
|\gx(x_i) - \gx(x_{i+1})|\leq 3\varepsilon.$ Since $\gx$ is linear on the interval
$[x_i ;x_{i+1}]$ we have $|\gx(x_i) - \gx(x)|\leq 3\varepsilon.$ 
If $|x_i - x_{i+1}|\geq 1/2^{n-1}$ then $\gx(x_i)=\gx(x)$. Similarly we show that
$|\hy(y_j) - \hy(y)|\leq 12\varepsilon.$ 

By point~(\ref{vv:GHOnGraph_JednaA}) we have
$|f(x_i,y_j)-\gx(x_i)-\hy(y_j)|\leq	4\varepsilon.$ Since $|(x,y)-(x_i,y_j)|<1/2^n\leq \delta$ it follows that
$|f(x,y)-f(x_i,y_j)|<\vare.$ Finally, $|f(x,y)-g(x)-h(y)|\leq 20\vare.$

The norms of the functions $\gx$ and $\hy$ are bounded because of point (\ref{vv:GHOnGraph_Dva}) of Theorem~\ref{vv:GHOnGraph}.
\end{proof}

\bibliographystyle{alpha}
\bibliography{Bibl_Trenklerova}

\end{document}